\documentclass[11pt,letterpaper]{amsart}

\usepackage{cmap}  
\usepackage[T2A]{fontenc}
\usepackage[utf8]{inputenc}
\usepackage[english]{babel}
\usepackage[usenames]{color}
\usepackage{amsmath,amssymb,amsthm}
\usepackage{arcs}
\usepackage{epsfig}
\usepackage{filecontents}
\usepackage{graphicx}
\usepackage[pdftex,colorlinks=true,linkcolor=blue,urlcolor=red,unicode=true,hyperfootnotes=false,bookmarksnumbered]{hyperref}
\usepackage{ifthen}
\usepackage{import}
\usepackage{indentfirst}
\usepackage{mathtools}
\usepackage{cancel}
\usepackage{xcolor}

\renewcommand{\le}{\leqslant}
\renewcommand{\ge}{\geqslant}

\newcommand{\N}{\mathbb{N}}

\newcommand{\ff}{\mathcal{F}}
\newcommand{\s}{\mathcal{S}}
\newcommand{\g}{\mathcal{G}}
\newcommand{\aaa}{\mathcal{A}}
\newcommand{\bb}{\mathcal{B}}

\newcommand{\h}{\mathcal H}

\newcommand{\aaaa}{\mathfrak{A}}

\newcommand{\T}{\mathcal{T}}
\newcommand{\W}{\mathcal{W}}

\newtheorem{thm}{Theorem}

\newtheorem{lem}[thm]{Lemma}
\newtheorem{cor}[thm]{Corollary}
\newtheorem{prop}[thm]{Proposition}
\newtheorem{conj}{Conjecture}
\newtheorem{obs}[thm]{Observation}

\title{An almost complete $t$-intersection theorem for permutations}

\date{}

\begin{document}
\maketitle
\begin{abstract}
For any $\epsilon>0$ and $n>(1+\epsilon)t$, $n>n_0(\epsilon)$ we determine the size of the largest $t$-intersecting family of permutations, as well as give a sharp stability result. This resolves a conjecture of Ellis, Friedgut and Pilpel (2011) and shows the validity of conjectures of
Frankl and Deza (1977) and Cameron (1988) for $n>(1+\epsilon )t$. We note that, for this range of parameters, the extremal examples are not necessarily trivial, and that our statement is analogous to the celebrated Ahlswede-Khachatrian theorem. The proof is based on the refinement of the method of spread approximations, recently introduced by Kupavskii and Zakharov (2022).
\end{abstract}
\section{Introduction}
Let $[n] = \{1,\ldots, n\}$ stand for the standard $n$-element set and let $2^{[n]}, {[n]\choose k}$ denote its power set and the set of all $k$-element subsets. One of the classical results in extremal combinatorics is the Erd\H os--Ko--Rado theorem \cite{EKR}. We say that a family of sets is {\it intersecting} if any two sets from the family have non-empty intersection. The EKR theorem states that for $n\ge 2k$ any intersecting family $\ff\subset {[n]\choose k}$ has size at most ${n-1\choose k-1}$. This result was highly influential, and by now grew into a subfield of extremal combinatorics, studying collections of objects with forbidden intersections.

Erd\H os, Ko and Rado \cite{EKR} also showed that the largest {\it $t$-intersecting} family of sets in ${[n]\choose k}$ has size ${n-t\choose k-t}$, provided $n>n_0(k)$. A family is {\it $t$-intersecting} if any two sets from the family intersect in at least $t$ elements. Note that the extremal example for this result, as well as the `classical' EKR theorem, is the family of all sets that contain a fixed $t$-element set. Later, Frankl \cite{F1} and Wilson \cite{W} determined the exact value of $n_0(k)$: the same conclusion holds for $n\ge (t+1)(k-t+1)$. For smaller values of $n$, other $t$-intersecting families become larger. For $i=0,\ldots, k-t$ define the Frankl families
$$\bb_i = \Big\{F\in {[n]\choose k}: |F\cap [t+2i]|\ge t+i\Big\}.$$
For $n<(t+1)(k-t+1)$ $|\bb_1|>|\bb_0|$. It turned out, however, that for any values of $n,k,t$ one of $\bb_i$ should be extremal. Frankl \cite{F1} introduced the families defined above and conjectured that for every $n,k,t$ one of them is extremal. This was shown for a wide range of parameters by Frankl and F\"uredi \cite{FF3} and then for all $n,k,t$ by Ahlswede and Khachatrian \cite{AK} in the so-called `Complete $t$-intersection Theorem'. This theorem played an important role in some applications to computer science, in particular the result of Dinur and Safra \cite{DS} from hardness of approximation.

Another very influential forbidden intersections result is due to Frankl and Wilson \cite{FW}. It addresses the so-called Erd\H os--S\'os problem: determine the largest family in ${[n]\choose k}$ with intersection exactly $t$ forbidden. Importantly, it gives an almost sharp result for the case when $k,t$ are linear in $n$ (under some number-theoretic restrictions). This was extremely important for applications, and the Frankl--Wilson theorem and a more general Frankl--R\"odl theorem \cite{FR} was used for several questions in discrete geometry  and Ramsey theory (for applications in discrete geometry, see \cite{FR2}, \cite{Rai1}, \cite{Rai2}).

Forbidden intersections were studied for structures other than families in $2^{[n]}$ or ${[n]\choose k}$: graphs \cite{EFF}, partitions \cite{MeMo} \cite{Kup54}, simplicial complexes \cite{Bor3} \cite{Kup55}, vector spaces \cite{FG}, and permutations. Permutations are by far the most studied object in this respect. Denote by $\Sigma_n$ the collection of permutations $[n]\to [n]$. The study of forbidden intersection theorems for permutations goes back to the paper of Frankl and Deza \cite{DeF}, in which they studied the question of how big a family of permutations from $\Sigma_n$ could be, if any two permutations $\sigma_1,\sigma_2$ from the family agree on at least $t$ points: satisfy $\sigma_1(x) = \sigma_2(x)$ for at least $t$ different $x\in[n]$. They showed that for $t=1$ the answer is $(n-1)!$, and showed that for $t=2,3$ and $n$ of certain form the answer is $(n-2)!$ and $(n-3)!$, respectively.\footnote{The reason behind the latter results is the existence of certain multiply transitive permutation subgroups for certain values of $n$, over which it is possible to do an averaging argument. It is one of the reasons why the algebraic combinatorics community got interested in the problem.} They conjectured that the answer is $(n-t)!$ for $n\ge n_0(t)$.\footnote{Frankl and Deza have also studied the dual question of extremal families with no two permutation satisfying $\sigma_1(x) = \sigma_2(x)$ for at least $t$ distinct $x\in[n]$ and established that the product of the two extremal functions (permutations with intersection at least $t$ and at most $t-1$)  is at most $n!$.}  For each $i=0,\ldots, \lfloor \frac {n-t}2\rfloor $ consider the following collections of families
\begin{align}\label{eqtintfam}  \mathfrak A_k =& \big\{\sigma \aaa_k\tau\ :\ \sigma,\tau\in \Sigma_n\big\}, \text{ where}\\
\notag \aaa_k:=&\big\{\sigma\in \Sigma_n: \sigma(i) = i \text{ for at least }t+k \text{ indices }i\in [t+2k]\big\}.\end{align}
Clearly, for each $i$ each family from $\mathfrak A_i$ is $t$-intersecting. The families from $\aaaa_0$ all have size $(n-t)!$ and thus provide a tight example for the Frankl--Deza conjecture.

The question of Frankl and Deza has a long history. First, it was proved by Cameron and Ku \cite{CK} and Larose and Malvenuto \cite{LM} that for $t=1$ the only extremal examples are of the form \eqref{eqtintfam}. Generalizing it to arbitrary $t$, Cameron conjectured that for $n\ge n_0(t)$ all extremal examples belong to $\aaaa_0$. In 2011, Ellis, Friedgut and Pilpel \cite{EFP} showed that the conjectures of Frankl and Deza and of Cameron hold (with the assumption  $n>n_0(t)$). They made a far-reaching conjecture in the spirit of the Complete $t$-Intersection Theorem for sets.

\begin{conj}\label{conjefp} For any $n,t$ the largest $t$-intersecting family of permutations must belong to one of the collections $\aaaa_k$.
\end{conj}
We should mention that a particular case of the conjecture was proved already by Frankl and Deza \cite{DeF} for $t$ being very close to $n$, i.e., for $n\ge n_0(n-t)$. In that regime, they showed that the family $\aaa_{\lfloor \frac{n-t}2\rfloor}$ is extremal.  In particular, comparing the sizes of $\aaa_i$ for different $i$, the conjecture suggests that $\aaa_0$ must be the largest $t$-intersecting family of permutations iff $n>2t$.\footnote{The size of any family in $\aaaa_1$ is $(t+2)(n-t-1)!-(t+1)(n-t-2)!$. The value $(n-t)!$ is bigger iff $(n-t-1)(n-2t-2)+(t+1)>0$.}  It has seen some great progress over the last several years. First, Ellis and Lifshitz \cite{EL} showed that the conjectures of Deza--Frankl, Cameron, and Ellis--Friedgut--Pilpel are valid for $t<c\frac{\log n}{\log\log n}$.  More recently, Zakharov and the author \cite{KuZa} managed to extend the validity of to $n>Ct/\log^2 t$ with some large $C$. Then Keller, Lifshitz, Minzer and Sheinfield \cite{KLMS} extended it to $n>Ct$ with some large $C$. In this paper, we prove these conjectures for $n>(1+\eta)t$ with arbitrary $\eta>0$.

\begin{thm}\label{thmbigt}
Let $\eta>0$ be some constant and $n,t$ be integers such that $n>(1+\eta)t$ and $n\ge n_0(\eta)$.  If $\ff\subset \Sigma_n$ is $t$-intersecting then $|\ff| \le \max_{k\in \N} |\aaa_k|$. Moreover, if $\ff$ is not fully contained in some $\ff\in \aaaa_k$ then $|\ff| \le (1-\frac 1e+o(1))\max_{k\in \N} |\aaa_k|$.
\end{thm}

Thus, this can be seen as the first Ahlswede-Khachatrian type result for permutations, in which families from $\aaaa_k$ for all $k$ appear.\footnote{We note without a proof that for constant $k$ and $n\to \infty$ the families $\aaaa_k$ are extremal in the interval $t\in [t_k,t_{k+1}]$, where $t_k=(1+o(1))\frac{k}{k+1}n$.}
Our theorem comes with a stability version, which is in fact tight, that extends similar results of \cite{EFP} and \cite{KLMS}. Many of the forbidden intersection results have stability versions: if a family has size close to extremal, then it has structure resembling that of the extremal example. Classical stability results on the EKR theorem are the Hilton--Milner theorem \cite{HM} and Frankl's degree theorem \cite{Fra4} (cf. also \cite{KZ2}).  Stability results on the Ahlswede-Khachatrian theorem include \cite{AK2}, \cite{Fried}. Ellis, Keller and Lifshitz \cite{EKL} managed to prove a stability result for the Erd\H os--S\'os problem.

Stability/structural results are often a way to get exact results. Once the rough structure of all nearly-extremal examples is determined, it often can be bootstrapped to obtain the extremal result. The classical Delta-system method, developed by Frankl and F\"uredi \cite{FF1, FF2} is one example of such approach. More recently, there was an impressive development of the methods coming from Boolean Analysis, based on junta approximations and hypercontractivity. Using juntas for the study of forbidden intersection problems started with the paper of Dinur and Friedgut \cite{DF} and then was greatly developed in the recent works of Ellis, Keller and Lifshitz \cite{EKL, EL, KL}. Hypercontractivity and sharp threshold results for the biased cube, as well as their applications to different extremal combinatorics questions were studied  in the recent paper by Keevash, Lifshitz, Long and Minzer \cite{KLLM}.

Another approach was suggested by Zakharov and the author \cite{KuZa}. It was developed around the recent breakthrough by Alweiss, Lovett Wu and Zhang \cite{Alw} on the Erd\H os--Rado sunflower conjecture. The approach is purely combinatorial, and, importantly, works with a large class of sufficiently quasirandom structures, such as partitions and permutations, which were previously studied mostly using algebraic techniques (Hoffman bounds combined with representation theory for permutations). Later, Keller, Lifshitz, Minzer and Sheinfeld \cite{KLM, KLMS} made a similar extension of the hypercontractivity approach.

We note that the authors of \cite{KLMS} also gave a two-family version of their result. It is also possible within our framework, but we decided to write the proofs for one family for conciseness.

\section{Proof of Theorem~\ref{thmbigt}}
Recall that $\Sigma_n$ is the set of permutations on $n$ elements. It will be convenient to view $\Sigma_n$ as a family of sets. One way is to treat it as an $n$-uniform subfamily in ${[n]^2 \choose n}$ by identifying a permutation $\sigma:[n]\to [n]$ with the set $\{(1,\sigma(1)),\ldots, (n,\sigma(n))\}\in {[n]^2\choose n}$. We will mostly use set terminology for permutations. We say that $S\subset [n]^2$ is a {\it partial permutation} if $S\subset \sigma$ for a permutation $\sigma\in\Sigma_n$. For any partial permutation $S$ we have $|\Sigma_n(S)| = (n-|S|)!$. If $S$ is not a partial permutation then $|\Sigma_n(S)| = 0$.
We denote by $\Sigma_n^{\le q}$ the family of partial permutations with $\le q$ elements. We use the following notation that is standard when working with families of sets. 
For sets $X\subset Y$ and a family $\mathcal S$ we put
\begin{align*}
\ff(X):=&\ \{F\setminus X:X\subset F, F\in \ff\},\\
\ff[X]:=&\ \{F: X\subset F, F\in \ff\},\\
\ff(X, Y):=&\ \{F\setminus X: F\cap Y = X\},\\
\ff[\s]:=&\ \bigcup_{A\in \s}\ff[A].
\end{align*}


\subsection{Sketch of the proof.} The rough structure of the proof is described below. We introduce and use a refinement of the spread approximation technique. Given the previous results, say, from \cite{KuZa}, we may for simplicity assume that $t \ge n^{1-\epsilon/2}$ for $\epsilon =0.01\eta$.   Let $\ff$ be the largest $t$-intersecting family of permutations.

 First, we find a family $\s$ of partial permutations of size somewhat smaller than $n$ (a concrete dependence is crucial here), such that it covers most of the extremal family $\ff$. At the same time, the family $\s$ is $t'$-intersecting for some $t'$ that is a bit smaller than $t$: $t' = t-n^{1-\epsilon/2}$. In this part, we trade $t$ for $t'$ in order to guarantee that the sets in $\s$ have a relatively small uniformity.

Second, we use the ``peeling'' technique to analyze the structure of $\s$. In this part, we use the condition on the uniformity of $\s$ in order to guarantee that the main contribution to the size of the family comes from partial permutations from $\s$ of size at most $t+k$, where $k = n^{1-\epsilon}$. This, in turn, allows to get a subset $X$ of size $t'$, such that $\ff(X)$ is rather dense in $\Sigma_n(X)$ (much denser than the original family is in $\Sigma_n$).

Third, we employ a greedy procedure that at each step uses the set $X$ as above, and manage to construct a much better spread approximation $\s'$ of the family. In particular, $\s'$ only contains sets of size $t+O(n^{1-\epsilon})$ and is $t$-intersecting.

The last step is to apply and analyze the peeling procedure for the refined approximation $\s'$. Within the procedure, we detect the most `important' layer and then show that this layer must be the only layer present in $\s'$ and thus in $\ff$, assuming that $\ff$ is extremal. After some analysis, this allows us to discern the structure isomorphic to $\aaa_k$ for some $k$.

\subsection{Preliminaries} We say that a family $\ff\subset 2^{[n]}$ is {\it $r$-spread}, if for any $X\subset [n]$ we have $|\ff(X)|\le r^{-|X|}|\ff|$. We begin with some simple observations concerning the property of being $r$-spread.

\begin{obs}\label{obs33}
  If for some $\alpha>1$ we have a family $\ff\subset {[n]\choose \le k}$ such that $|\ff|\ge \alpha^k$, then there is a set $X$ such that $|\ff(X)|>1$  and $\ff(X)$ is $\alpha$-spread.
\end{obs}
\begin{proof} if $\ff$ is $\alpha$-spread then we are done. Otherwise, find an inclusion-maximal set $X$ such that $|\ff(X)|> \alpha^{-|X|}|\ff|$. Note that, by $|X|\le k$ and $|\ff|\ge \alpha^k$ we have $|\ff(X)|>1$. We claim that $\ff(X)$ is $\alpha$-spread. Indeed, arguing indirectly, assume that there is a nonempty set $Y$, $Y\cap X = \emptyset$, such that $|\ff(X\cup Y)|>\alpha^{-|Y|}|\ff(X)|> \alpha^{-|X|-|Y|}|\ff|$. But then $Y\cup X$ is a valid choice for $X$, which was supposed to be inclusion-maximal.
\end{proof}
Arguing as in the proof of Observation~\ref{obs33}, we get the following observation.
\begin{obs}\label{obs34}
  Fix $\alpha>1$ and consider a family $\ff$. If $X$ is an inclusion-maximal set with the property that $\ff(X)\ge \alpha^{-|X|}|\ff|$, then $\ff(X)$ is $\alpha$-spread.
\end{obs}

We say that $W$ is a {\it $p$-random subset} of $[n]$ if each element of $[n]$ is included in $W$ with probability $p$ and independently of others. The spread approximation technique is based on the following theorem.
\begin{thm}[\cite{Alw}, a sharpening due to \cite{Tao}]\label{thmtao}
  If for some $n,k,r\ge 1$ a family $\ff\subset {[n]\choose \le k}$ is $r$-spread and $W$ is an $(\beta\delta)$-random subset of $[n]$, then $$\Pr[\exists F\in \ff\ :\ F\subset W]\ge 1-\Big(\frac 5{\log_2(r\delta)} \Big)^\beta k.$$
\end{thm}

In the remainder of this subsection, we will prove an estimate on the size of $\aaa_k$. 
We will need the following inequality, valid for any positive integers $n,x$.
\begin{equation}\label{eqfac}
  \frac{(n+x)!}{n!}=\prod_{i=n+1}^{n+x}i\ge \big((n+x)!\big)^{\frac x{n+x}}\ge \Big(\Big(\frac{n+x}e\Big)^{n+x}\Big)^{\frac x{n+x}}=\Big(\frac {n+x}e\Big)^x.
\end{equation}

In what follows, we bound the sizes of $\aaa_i$. For an integer $\ell\in[n]$, we say that a permutation $\sigma\in \Sigma_n$ is an $\ell$-derangement, if $\sigma(i)\ne i$ for each $i\in [\ell]$.
\begin{lem}\label{lemsize}
 For any $n,t,k$ we have $|\aaa_k| \le {t+2k\choose k}(n-t-k)!$. If $k = o(n-t-k)$ then we have $|\aaa_k| = (1+o(1)){t+2k\choose k}(n-t-k)!$. 
\end{lem}
\begin{proof}
  Every permutation in $\aaa_k$ can be described by first choosing some $(t+k)$-element subset of $[t+2k]$ that it fixes, and then defining it on the remainder (that is, choosing a way how to permute the remaining $n-t-k$ elements). Therefore, $|\aaa_k|\le {t+2k\choose k}(n-t-k)!$ for any values of $n,t,k$.

  Let us give a lower bound on $|\aaa_k|$ for $k=o(n-t-k)$. Clearly, $|\aaa_k|$ is at least the number of permutations that fix exactly $t+k$ elements on $[t+2k]$. For any fixed subset $S\in {[t+2k]\choose t+k}$, the number of permutations that fix elements from $S$ and no other element of $[t+2k]$ is the same as the number of $k$-derangements in $\Sigma_{n-t-k}$. Using inclusion-exclusion formula, this number is
  \begin{equation}\label{eqkderan}\sum_{i=0}^k(-1)^i {k\choose i}(n-t-k-i)! = (1+o(1))(n-t-k)!,\end{equation}
  where the equality uses the fact that $k=o(n-t-k)$. The collections of permutations that fix exactly $S_1$, $S_2$ are disjoint for distinct $S_1,S_2\in {[t+2k]\choose t+k}$. Thus, we get that $|\aaa_k|\ge (1+o(1)){t+2k\choose k}(n-t-k)!$.
\end{proof}

\subsection{First spread approximation}  We begin the proof of the main result by constructing the first, `rough' spread approximation. In the next stage, we shall repeatedly apply this theorem to $\ff$ and some of its `dense' subfamilies.

\begin{thm}\label{thm2}
  Let $n\ge 2$, $t\ge1$ be some integers and fix a real number $\epsilon\in (0,1/2)$, such that $n\ge (1+2\epsilon)t$ and $t\ge n^{1-\epsilon/2}$. Suppose that $n\ge n_0(\epsilon)$.  Consider a family $\ff\subset \Sigma_n$ that is $t$-intersecting.   
  Put $t'= t-n^{1-\epsilon/2}$. Then there exist a  $t'$-intersecting family $\mathcal S\subset \Sigma_n^{\le (1+\epsilon)t}$ of partial permutations of size at most $(1+\epsilon)t$ and a family $\ff'$ such that the following holds.
  \begin{itemize}
    \item[(i)] $\ff\setminus \ff'\subset \Sigma_n[\s]$;
    \item[(ii)] for any $B\in \s$ there is a family $\ff_B\subset \ff$ such that $\ff_B(B)$ is $n^{\epsilon/2}$-spread;
    \item[(iii)] $|\ff'|\le n^{-\frac 15\epsilon t}(n-t)!$.
  \end{itemize}
\end{thm}
\begin{proof}[Proof of Theorem~\ref{thm2}] We construct $\mathcal S$ as follows. Put $r:=
n^{\epsilon/2}$ and consider the following procedure for $i=1,\ldots $ with $\ff^1:=\ff$.
\begin{enumerate}
    \item Find a maximal $S_i$ that  $|\ff^i(S_i)|\ge  r^{-|S_i|}|\ff^i|$.
    \item If $|S_i|> (1+\epsilon)t$ or $\ff^i = \emptyset$ then stop. Otherwise, put $\ff^{i+1}:=\ff^i\setminus \ff^i[S_i]$.
\end{enumerate}
Note that Observation~\ref{obs34} and maximality of $S_i$ imply that $\ff^i(S_i)$ is $r$-spread. 
Let $N$ be the step of the procedure for $\ff$ at which we stop. The family $\s$ is defined as follows: $\s:=\{S_1,\ldots, S_{N-1}\}$. Clearly, $|S_i|\le (1+\epsilon)t$ for each $i\in [N-1]$. The family $\ff_{B}$ promised in (ii) is defined to be $\ff^i[S_i]$ for $B=S_i$. Next, note that if $\ff^N$ is non-empty, then 
\begin{align*}|\ff^N|\le&\ r^{|S_N|}  |\ff^{N}(S_N)|\le  r^{|S_N|}(n-|S_N|)!\\ 
\overset{\eqref{eqfac}}{\le}&\
n^{\frac 12\epsilon|S_N|}\Big(\frac{n-t}e\Big)^{-(|S_N|-t)}(n-t)!\\ \le&\ n^{\frac 12\epsilon|S_N|}\Big(\frac{2\epsilon n}e\Big)^{-(|S_N|-t)}(n-t)!\\
\le&\ n^{\frac 12\epsilon (1+\epsilon)t}\Big(\frac{2\epsilon n}e\Big)^{-\epsilon t}(n-t)!\\
\le&\ (2\epsilon/e)^{-\epsilon t} n^{-\frac 14\epsilon t}(n-t)!\le n^{-\frac 15\epsilon t}(n-t)!\end{align*}
The inequality in the fourth line uses the fact that the expression in the third line is a decreasing function of $|S_N|$, provided $2\epsilon n\ge e n^{\epsilon n/2}$, and that $|S_N|\ge (1+\epsilon)t$. We put $\ff':=\ff^N$. Since either $|S_N|>(1+\epsilon)t$ or $\ff' = \emptyset $, we have $|\ff'|\le n^{-\frac 15\epsilon t}(n-t)!$. We are only left to verify that $\s$ is $t'$-intersecting. We prove it below.
\begin{lem}\label{lemtint} The family $\mathcal S$ is $t'$-intersecting. \end{lem}
\begin{proof}

  Take two (not necessarily distinct) $A_1,A_2\in \mathcal S$  and assume that $|A_1\cap A_2|<t'$. Recall that the families $\ff_{A_i}(A_i)$, $i=1,2$ are both $r$-spread with $r = n^{\epsilon/2}$. 
  For $i\in [2]$ put $j:=3-i$ and put
  $$\g_i:=\ff_{A_i}(A_i)\setminus\Big\{F\in \ff_{A_i}(A_i): |F\cap A_j\setminus A_i|\ge \frac{n^{1-\epsilon/2}}2\Big\}.$$ The size of the latter family is at most
 \begin{align*}{|A_j\setminus A_i|\choose n^{1-\epsilon/2}/2}\max_{X: |X| = n^{1-\epsilon/2}/2, X\cap A_i = \emptyset}|\ff_{A_i}(A_i\cup X)|\le&\\  {n\choose n^{1-\epsilon/2}/2}n^{-\frac\epsilon 2 n^{1-\epsilon/2}}|\ff_{A_i}(A_i)|\le&\\
  (2e n^{\epsilon/2})^{\frac 12n^{1-\epsilon/2}}n^{-\frac\epsilon 2 n^{1-\epsilon/2}}|\ff_{A_i}(A_i)|\le&\ \frac 12 |\ff_{A_i}(A_i)|.\end{align*}
  This implies that $|\g_i|\ge \frac 12 |\ff_{A_i}(A_i)|$.
 Because of this and the trivial inclusion $\g_i(Y)\subset \ff_{A_i}(A_i\cup Y)$, valid for any $Y$,  we conclude that $\g_i$ is $ \frac r2$-spread for both $i\in [2]$. Since $n$ is sufficiently large, we have $\frac r2 > 2^{11}\log_2(2n)$.  

  What follows is an application of Theorem~\ref{thmtao}. Let us put $\beta= \log_2(2n)$ and $\delta = (2\log_2(2n))^{-1}$. Note that $\beta\delta = \frac 12$ and $\frac r2\delta > 2^{10}$ by our choice of $r$.  Theorem~\ref{thmtao} implies that a $\frac{1}2$-random subset $W_i$ of $[n]\setminus A_i$ contains a set from $\g_i$ with probability strictly bigger than
  $$1-\Big(\frac 5{\log_2 2^{10}}\Big)^{\log_2 (2n)} n = 1-2^{-\log_2 (2n)} n = \frac 12.$$

  Consider a random partition of $[n]\setminus (A_1\cup A_2)$ into $2$ parts $U_1',U_2'$, where $\Pr[x\in U_i']=1/2$ independently for each $i\in [2]$ and $x\in [n]\setminus (A_1\cup A_2)$. Next, for $i\in[2]$ and $j=3-i$, put $U_i:= U_i'\cup (A_j\setminus A_i)$. Then  $U_i$ is a random subset of the same ground set as $W_i$ above, moreover, the probability of containing each element is at least that for $W_i$. It implies that there is $F_i \in \g_i$ such that $F_i\subset U_i$ with probability strictly bigger than $\frac 12$. Using the union bound, we conclude that, with positive probability, it holds that there are such $F_i$, $F_i\subset U_i,$ for both  $i \in[2]$. Fix such choices of $U_i$ and $F_i$, $i \in [2]$. Then, on the one hand, each $F_i\cup A_i$ belongs to $\ff$ and, on the other hand, $|(F_1\cup A_1)\cap (F_2\cup A_2)| =|A_1\cap A_2|+|F_1\cap A_2|+|F_2\cap A_1|< |A_1\cap A_2|+n^{1-\epsilon/2}<t$. The first inequality follows from the definition of $\g_i$, and the second inequality due to $|S_1\cap S_2|<t'$ and the definition of $t'$. This is a contradiction with $\ff$ being $t$-intersecting.
  \end{proof}
This concludes the proof of Theorem~\ref{thm2}.\end{proof}

\subsection{The peeling procedure and the structure of spread approximations for large $t$-intersecting families}\label{sec21}

In this section, we analyze the `peeling' procedure that captures the structure of the $t$-intersecting families of partial permutations. One of the key objectives is to control the number of permutations coming from higher uniformity partial permutations.

 We say that a $t$-intersecting family $\T$ is {\it maximal} if for any $T \in \T$ and any proper subset $X \subsetneq T$ there exist $T' \in \T$ such that $|T\cap T'|<t$ and, moreover, $T_1\not\subset T_2$ for $T_1,T_2\in \T$ (i.e., $\T$ is an {\it antichain}).

\begin{obs}\label{obs22}
For any positive integers $n,p$ and a $t$-intersecting family $\s \subset \Sigma^{\le q}_n$, there exists a maximal $t$-intersecting family $\T$ 
 such that for every $T\in \T$ there is $S\in \s$ such that $T\subseteq S$. Consequently, for any family $\ff\subset \Sigma_n$ we have $\ff[\s] \subset \ff[\T]$.
\end{obs}
The easiest way to see it is to construct $\T$ by repeatedly replacing sets in $\s$ by their proper subsets while preserving the $t$-intersecting property.

Consider a $t$-intersecting family $\s\subset \Sigma^{\le q}_n$. Let us iteratively define the following series of families.
\begin{enumerate}
    \item Put $\T_{q-t}=\s$. 
    \item For $k = q-t,q-t-1, \ldots, 0$ we put $\W_k = \T_k \cap {[n] \choose t+k}$ and let $\T_{k-1}$ be the family given by Observation~\ref{obs22} when applied to 
        $\T_{k}\setminus \W_{k}$ playing the role of $\mathcal S$.
\end{enumerate}
Thus, we `peel' the sets of the largest uniformity in $\T_k$ and replace the resulting family by a maximal $t$-intersecting family using Observation~\ref{obs22}.  Remark that $\T_k$ is $t$-intersecting for each $k=q-t,q-t-1\ldots,0$ by definition. We summarize the basic properties of these series of families in the following lemma.

\begin{lem}\label{lemkeyred} The following properties hold for each $k = q-t,q-t-1\ldots, 0$. 
\begin{itemize}
  \item[(i)] All sets  in $\T_k$ have size at most $t+k$.
  \item[(ii)] For $k\le q-t-1$ we have $\aaa[\T_{k+1}]\subset \aaa[\T_{k}]\cup \aaa[\W_{k+1}]$.
  \item[(iii)] There is no set $X$ of size $\le t+k-1$ 
      such that $\W_k(X)$ (or any subfamily of $\T_k(X)$) is $\alpha$-spread for $\alpha>k$. 
\end{itemize}
\end{lem}
\begin{proof}
(i) This easily follows by reverse induction on $k$. The base case is that all sets in $\s$ have size at most $q$, and the induction step is via the definition of $\T_{k-1}$.

(ii) We have $\aaa[\T_{k+1}] = \aaa[\T_{k+1} \setminus \W_{k+1}] \cup \aaa[\W_{k+1}]$ and, by the definition of $\T_{k}$, we have $\aaa[\T_k]\supset \aaa[\T_{k+1}\setminus \W_{k+1}]$.

(iii) Assume that there is such a set $X$. By the maximality of $\T_k$, there must be a set $F\in \T_k$ such that $|X\cap F|=z<t$.
At the same time, the family $\T_k$ is  $t$-intersecting. So, for each set  $G\in \W_k[X]$ we have $|G\cap F|\ge t$, and thus $|F\cap (G\setminus X)| \ge t-z$ for each such $G$. Thus, $\W_k[X]\subset \cup_{S\in {F\choose t-z}}\W_k[X\cup S]$ and so, using the $\alpha$-spreadness of $\W_k(X)$, we get \begin{align*}|\W_k(X)|\le&\ {|F|-z\choose t-z}\max_{F'\in {F\choose t-z}}|\W_k(X\cup F')|\\
\le&\ {|F|-z\choose t-z}\alpha^{-(t-z)}|\W_k(X)| \\
=&\ \prod_{i=1}^{t-z}\frac{|F|-t+i}i\alpha^{-(t-z)}|\W_k(X)|\\
\le&\ (|F|-t+1)^{t-z}\alpha^{-(t-z)}|\W_k(X)|<|\W_k(X)|,\end{align*}
where the last inequality is due to $|F|-t+1\le k<\alpha$. This a contradiction.
\end{proof}

Our next goal is to bound the size of $\W_k$.
Consider two sets $A,B\in \T_k$ such that $A\cap B = I$, $|I|=t$. (There are such sets by maximality of $\T_k$.) Recall that the sizes of sets of $A,B$ are at most $t+k$. 
Any set $C\in \W_k$ intersects both $A$ and $B$ in at least $t$ elements. Denote $C_0=C\cap I$, $C_1=C\cap (A\setminus I)$ and $C_2 = C\cap (B\setminus I)$. Then, putting $|C_0|=t-j$, we must have $|C_1|, |C_2|\ge j$. Select $C_1'\subset C_1,C_2'\subset C_2$ such that $|C_1'| = |C_2'| = j$. The family $\W_k(C_0\cup C_1'\cup C_2'\cup X)$ by Lemma~\ref{lemkeyred} (iii) is not  $\alpha$-spread subfamily for $\alpha>k$ and any $X$, and thus $|\W_i(C_0\cup C_1'\cup C_2')|\le k^{k-j}$ by Observation~\ref{obs33}. From here, we see that we can upper bound the number of sets in $\W_k$ as follows:
$$|\W_k|\le \sum_{j=0}^{k}{t\choose t-j}{|A|-t\choose j}{|B|-t\choose j} k^{k-j}\le \sum_{j=0}^{k}f(j),$$
where
\begin{equation}\label{eqfj} f(j):= {t\choose t-j}{k\choose j}^2 k^{k-j}.\end{equation}
Note that
$$\frac {f(j)}{f(j+1)} = \frac {j+1}{t-j}\Big(\frac {j+1}{k-j}\Big)^2 k = \frac{(j+1)^3 k}{(t-j)(k-j)^2}.$$
The observation below is based on simple calculations that check whether the last ratio above is bigger than $1$ or not.
\begin{obs}\label{obs7} For sufficiently large $t$, $t\ge k$, we have the following. Denote by $j_0$ the value at which $f(j)$ attains its maximum.
  \begin{enumerate}
    \item For $k$ such that $t=o(k^2)$, we have $j_0 =(1+o(1))(tk)^{1/3}$.
    \item For $k=\Theta (\sqrt t)$, we have  $j_0,k-j_0 = \Theta (k)$.
    \item For $k = o(\sqrt t)$ we have  $k-j_0 = o(k)$.
    \item For $k = o(t^{1/4})$, we have  $j_0 = k$.
  \end{enumerate}
  Moreover, $|\W_i|\le (k+1)f(j_0)$ in general and $|\W_i|\le (1+o(1))f(k) = (1+o(1)){t\choose t-k}$ for $k = o(t^{1/4})$.
\end{obs}
We will need the following concrete statement concerning the size of $\Sigma_n[\W_k].$

\begin{lem}\label{lempeelbound}
  For any $\epsilon>0$, $n\ge (1+100\epsilon)t$, $t>t_0(\epsilon)$ and $t^{0.01}\le k\le 2\epsilon t$ we have
  \begin{equation}\label{eqobs7}\frac{|\W_k|(n-t-k)!}{(n-t)!}\le 2^{-k}.\end{equation}
\end{lem}
\begin{proof}
  Using \eqref{eqfj} and the inequality ${m\choose \ell}\le (em/\ell)^\ell$, we get
  \begin{equation}\label{eqfj2} f(j)\le \Big(\frac{e^3tk^2}{j^3}\Big)^jk^{k-j}.\end{equation}
Let us first consider the case when $k\ge t^{0.99}$. Then we can apply Observation~\ref{obs7} (1) and get that the fraction in \eqref{eqfj2} for $f(j_0)$ is simplified to $((1+o(1))e^3k)^{j_0}$. We thus get
\begin{align*}  \frac{|\W_k|(n-t-k)!}{(n-t)!}\le& (k+1) \frac{f(j_0)(n-t-k)!}{(n-t)!} \\
\le&  \frac{((1+o(1))e^3k)^{j_0}k^{k-j_0+1}}{(n-t-k)^k}
\le  \frac{((1+o(1))e^3k)^k}{(n-t-k)^k} \le \ 2^{-k},\end{align*}
where the last inequality is due to $n-t-k\ge 98\epsilon t\ge 49k> 2.1 e^3k$.

  Next, assume that $t^{0.01}\le k\le t^{0.99}$. In this regime, according to Observation~\ref{obs7}, we have $tk^2/j_0^3\le (1+o(1))t^{0.99}$, and thus we can conclude that \eqref{eqfj2} implies the following inequality.
  $$f(j)\le (Ct^{0.99})^k,$$
  where $C$ is some absolute constant. In this regime, we can conclude the following.
  \begin{align*}  \frac{|\W_k|(n-t-k)!}{(n-t)!}\le& (k+1) \frac{f(j_0)(n-t-k)!}{(n-t)!} \\
\le&    \frac{(k+1)(Ct^{0.99})^k}{(n-t-k)^k} \le \ 2^{-k},\end{align*}
where in the last inequality we use that $n-t-k\ge \epsilon t$ and the fact that $t>t_0(\epsilon)$.
\end{proof}

\subsection{Finding a dense piece of $\ff$ and a sharper spread approximation}
We return to the notation used in Theorem~\ref{thm2}. The following theorem makes use of the peeling procedure and allows us to find a partial permutation $X$ such that the family $\ff(X)$ is much denser in $\Sigma_n(X)$ than $\ff$ in $\Sigma_n$.
\begin{thm}\label{thmstage1} Assume that $\epsilon\in(0,1/2)$ and $n,t$ are positive integers such that $n>(1+100\epsilon)t$ and $n>n_0(\epsilon)$, $t\ge 2n^{1-\epsilon/2}$.  Suppose that $\ff\subset \Sigma_n$ is $t$-intersecting and satisfies $|\ff|\ge 2^{2-x}(n-t')!$ for some integer $x$, $t^{0.01}<x\le\frac 15\epsilon t$ , and where $t' = t- n^{1-\epsilon/2}$. Then there is a subset $X$ of size $t'$ such that $|\ff[X]|\ge \frac 12 {t'+x\choose t'}^{-1}|\ff|$.
\end{thm}

 \begin{proof}[Proof of Theorem~\ref{thmstage1}]
Take a family $\ff$ as in the statement. First, we use Theorem~\ref{thm2} to get a spread approximation $\s$ of sets of size at most $q=(1+\epsilon)t$, such that $\s$ is $t'$-intersecting, and a remainder $\ff'$, $|\ff'|\le n^{-\frac 15 \epsilon t} (n-t)!$, that satisfy $\ff\setminus \ff'\subset \Sigma_n[\s]$. (It is easy to see that all requirements of Theorem~\ref{thm2} are met.)

Next, we apply the peeling procedure from Section~\ref{sec21} to $\s$ with $t'$ playing the role of $t$, obtaining a sequence of families $\W_k$. We  want to bound the contribution of $|\ff[\W_k]|$ with large $k$ to $|\ff|$. To this end, we apply Lemma~\ref{lempeelbound} with $t'$ playing the role of $t$ for each $k= x+1,\ldots, q-t'=\epsilon t+(t-t')$. Note that all conditions are satisfied, most importantly, $k\le \epsilon t+(t-t')\le 2\epsilon t$. Then \eqref{eqobs7} implies the following bound on the contribution of layers $\W_k$ to $|\ff|$ for $k\ge x+1$.
$$\sum_{k=x+1}^{q-t'}|\ff[\W_k]| \le \sum_{k=x+1}^{q-t'}|\Sigma_n[\W_k]| \le (n-t')!\sum_{k=x+1}^{q-t'}\frac{|\W_{k}|(n-t'-k)!}{(n-t')!} \le $$
$$(n-t')! \sum_{k=x+1}^{q-t'} 2^{-k}\le (n-t')!\cdot 2^{-x}.$$

Put $\ff'' =(\ff\setminus \ff')\cap \Sigma_n[\T_{t'+x}]$. (In words, we removed from $\ff$ the remainder $\ff'$ and `peeled off' the layers $\W_k$ of uniformity at least $t'+x+1$.) We have $|\ff|> 4(n-t')!\cdot 2^{-x}$, and thus $|\ff\setminus\ff'|\ge \frac 34|\ff|$ and  $|\ff''|\ge \frac 12 |\ff|$. Take any set $F\in \T_{t'+x}$. Via averaging, there is a $t'$-element subset $X\subset F$, such that $|\ff''(X)|\ge {t'+x\choose t'}^{-1}|\ff''|\ge \frac 12 {t'+x\choose t'}^{-1}|\ff|$.
\end{proof}

The following corollary shall be the building block for a better spread approximation result.

\begin{cor}\label{corstage1}
Assume that $\epsilon\in(0,1/2)$ and $n,t$ are positive integers such that $n>(1+100\epsilon)t$, $n>n_0(\epsilon)$, and $t\ge 2n^{1-\epsilon/2}$.  Suppose that $\ff\subset \Sigma_n$ is $t$-intersecting and satisfies $|\ff|\ge (n-t)! \cdot 2^{-\frac 12 n^{1-\epsilon/4}}$. Then there is a set $X$, $|X|\le t+n^{1-\epsilon/6}$, such that $\ff(X)$ is $10\epsilon t$-spread.
\end{cor}

\begin{proof}
  We apply Theorem~\ref{thmstage1} with $x = n^{1-\epsilon/4}$. Then $$4(n-t')!\cdot 2^{-x} \le  (n-t)!\cdot 4n^{n^{1-\epsilon/2}}2^{-n^{1-\epsilon/4}}\le (n-t)! \cdot 2^{-\frac 12 n^{1-\epsilon/4}},$$
and  the condition on $|\ff|$ in Theorem~\ref{thmstage1} is implied by $|\ff|\ge (n-t)! \cdot 2^{-\frac 12 n^{1-\epsilon/4}}$.
Theorem~\ref{thmstage1} gives us a $t'$-element set $X$, such that
$$|\ff(X)|\ge \frac 12 {t'+x\choose x}^{-1}|\ff|\ge n^{-x}|\ff|\ge (n-t')!\cdot 2^{-x} n^{-x}\ge (n-t')!\cdot 2^{-n^{1-\epsilon/5}}.$$
We claim that there is a set $Y$ satisfying $|Y|\le t'+n^{1-\epsilon/6}$ and $X\subset Y$, such that, moreover, $\ff(Y)$ is $10\epsilon t$-spread. Indeed, find a maximum set $Y$, $Y\supset X$, that violates the $10\epsilon t$-spreadness of $\ff(X)$. If $|Y|> t'+n^{1-\epsilon/6}$, then, using \eqref{eqfac},
$$|\ff(X)|\le (10\epsilon t)^{|Y|-t'}(n-|Y|)!\overset{\eqref{eqfac}}{\le} (10\epsilon t)^{|Y|-t'}\Big(\frac{n-t'}e
\Big)^{|Y|-t'} (n-t')!\le $$
$$
2^{-|Y|+t'}(n-t')!<2^{-n^{1-\epsilon/6}}(n-t')! <|\ff(X)|, $$
which is a contradiction. Here in the third inequality we used $n-t'\ge 100\epsilon t$.
\end{proof}

Using this corollary, we can get a much better spread approximation of our family.

\begin{thm}\label{thmstage2}
  Let $n\ge 2$, $t\ge1$ be some integers and fix a real number $\epsilon\in (0,1/2)$, such that $n>n_0(\epsilon)$, $n>(1+100\epsilon)t$, $t\ge 2n^{1-\epsilon/2}$.  Consider a family $\ff\subset \Sigma_n$ that is $t$-intersecting.   
  Then there exist a  $t$-intersecting family $\mathcal S\subset \Sigma_n^{\le t+n^{1-\epsilon/6}}$ of partial permutations of size at most $t+n^{1-\epsilon/6}$ and a family $\ff'\subset \ff$ such that the following holds.
  \begin{itemize}
    \item[(i)] $\ff\setminus \ff'\subset \Sigma_n[\s]$;
    \item[(ii)] for any $B\in \s$ there is a family $\ff_B\subset \ff$ such that $\ff_B(B)$ is $10\epsilon t$-spread;
    \item[(iii)] $|\ff'|\le 2^{-\frac 12 n^{1-\epsilon/4}}(n-t)!$.
  \end{itemize}
\end{thm}

\begin{proof}[Proof of Theorem~\ref{thmstage2}]
We iteratively apply Corollary~\ref{corstage1}. We put $\ff_0 :=\ff$, and at the $i$-th step we get a set $Y$ and a family $\ff_i[Y]$. We then put $\ff_{i+1}:=\ff_i\setminus  \ff_i[Y]$ and apply Corollary~\ref{corstage1} to $\ff_{i+1}$ as long as the corollary is applicable, that is, until step $N$ when $|\ff_N|\le (n-t)! \cdot 2^{-\frac 12 n^{1-\epsilon/4}}$ for some $N$. We then put $\ff':=\ff_N$. The collection of $Y$'s accumulated while running the procedure is our family $\mathcal S$, and we put $\ff_Y[Y]$ to be the corresponding family $\ff_i[Y]$.

The only property that we are left to verify is that $\s$ is $t$-intersecting. This is done in a way that is very similar to the proof of Lemma~\ref{lemtint}, and we sketch it below. Arguing indirectly, take $A_1,A_2\in \s$ that are not $t$-intersecting and, putting $x:=t-|A_1\cap A_2|$, for each $i\in [2]$ and $j=3-i$ define
 $$\g_i:=\ff_{A_i}(A_i)\setminus\Big\{F\in \ff_{A_i}(A_i): |F\cap A_j\setminus A_i|\ge x\Big\}.$$
 Using the bound on $|A_i|$ and the spreadness of $\ff_{A_i}(A_i)$, the size of the subtracted family is at most
 \begin{align*}{|A_j\setminus A_i|\choose x}\max_{X: |X| = x, X\cap A_i = \emptyset}|\ff_{A_i}(A_i\cup X)|\le&\\   {n^{1-\epsilon/6}+x\choose x}(10\epsilon t)^{-x}|\ff_{A_i}(A_i)|\le&\\
  (n^{1-\epsilon/6}+1)^x(10\epsilon t)^{-x}|\ff_{A_i}(A_i)|\le&\ \frac 12 |\ff_{A_i}(A_i)|.\end{align*}
  We conclude that $|\g_i|\ge \frac 12 |\ff_{A_i}(A_i)|$ and, consequently, that $\g_i$ is $5\epsilon t$-spread for both $i\in [2]$. Since $n$ is sufficiently large and $t>2n^{1-\epsilon/2}$, we have $5\epsilon t > 2^{11}\log_2(2n)$, and we can apply the coloring argument, finding two disjoint sets $U_1,U_2$, where $U_i\in \g_i$. Then $U_i\cup A_i$ violate the $t$-intersection property of $\ff$.
\end{proof}

\subsection{Proof of Theorem~\ref{thmbigt}} 
Fix $\eta>0$ and consider a family $\ff\subset \Sigma_n$ of permutations that is $t$-intersecting and that has size at least half of the largest such family. If $t<2n^{1-0.005\eta}$ then we are done by the result of \cite{KuZa}, so we assume that the opposite inequality holds.\footnote{We could have incorporated this case into the framework of this paper, but decided to omit it to reduce technicalities.} Apply Theorem~\ref{thmstage2} with $\epsilon= 0.01\eta$, getting a family $\s$ of sets of size at most $q:=t+n^{1- \epsilon/6}$ and a small remainder $\ff'\subset \ff$. (We note that the conditions on $t$ and $n$ for such choice of $\epsilon$ are implied by the condition on $n$ in Theorem~\ref{thmbigt} and our assumption on $t$ above.) Next, apply the peeling procedure to $\s$. Fix $k$ to be the largest value such that, first, $k\le t^{0.01}$ and, second, $|\W_k|\ge t^{-0.5}{t\choose k}$. We stop the peeling procedure at $k$ and shall show that most of the family $\ff$ is contained in $\Sigma_n[\W_k]$. Recall that $|\Sigma_n(\W_k)|\le |\W_k|(n-t-k)!$.

First, we note that, by \eqref{eqobs7}, as in the proof of Theorem~\ref{thmstage1}, we have \begin{equation}\label{eqres1}
\sum_{x = t^{0.01}+1}^{q-t}|\Sigma_n(\W_x)|\le \sum_{x = t^{0.01}+1}^{q-t}2^{-x}(n-t)!\le 2^{-t^{0.01}}(n-t)!.                                        \end{equation}
Second, using the inequality on $|\W_x|$ for $x>k$ and Lemma~\ref{lemsize}, we have
\begin{align}\notag \sum_{x = k+1}^{t^{0.01}}|\Sigma_n(\W_x)|\le & \sum_{x = k+1}^{t^{0.01}}\frac{|\W_k|(n-t-k)!}{(1+o(1)){t+2k\choose k}(n-t-k)!}|\aaa_x|\\
\label{eqres2} \le& \sum_{x = k+1}^{t^{0.01}}(1+o(1))t^{-0.5}|\aaa_x|\le t^{-0.4}\max_x |\aaa_x|.\end{align}

For the subsequent analysis, we need to better understand the structure of $\W_k$. Looking at the formula \eqref{eqfj} and the displayed equation above it, it should be clear that, (assuming $k<t^{0.01}$) for every $j\le k-1$ we have $f(j)=o\big(t^{-0.9}f(k)\big)= o\big(t^{-0.9}{t\choose t-k}\big)$. Moreover, if any two sets in $\W_k$ intersect in at least $t+1$ elements, then we can do the same analysis and calculations as the one that preceded \eqref{eqfj}, but with $A,B\in \W_k$ and with $I$ of size at least $t+1$ and $C$ intersecting $A,B$ in at least $t+1$ elements, and get a bound $|\W_k|=o\big(t^{-0.9}{t\choose t-k}\big)$. Thus, we conclude that, first, we have two sets $A,B\in \W_k$ such that $|A\cap B| = t$, and, moreover, by Observation~\ref{obs7}, most of the sets from $\W_k$ must intersect $U_0 = A\cap B$ in exactly $t-k$ elements and thus contain $A\Delta B$. 
  Let us fix one such set $C$ and put $I = A\cap B\cap C$. Let $I'$ stand for the set of elements that belong to exactly two out of $A,B,C$. That is, $A\cup B\cup C = I\sqcup I'$.

We are ready to complement \eqref{eqres1}, \eqref{eqres2} with the analysis of the sizes of layers below $k$. Put $\g_j:=\T_k\cap {[n]^2\choose t+j}$ for $0\le j<k$.
First and foremost, let us bound the size of the intersection of $G\in \g_j$ with $I$. The set $G$ intersects each of $A,B,C$ in at least $t$ elements, thus $3t$ elements in total (with multiplicities). By double-counting, knowing that elements in $I$ have multiplicity $3$ and the elements in $I'$ have multiplicity $2$, we must have
$$|G|\ge |G\cap I|+|G\cap I'|\ge |G\cap I|+\Big\lceil\frac 32 (t-|G\cap I|)\Big\rceil= t+\Big\lceil\frac 12(t-|G\cap I|)\Big\rceil.$$ We have $|G|=t+j$, and  thus $t-|G\cap I|\le 2j$. In other words, $|G\cap I|\ge t-2j = |I|+k-2j$. From here, we can conclude that
$$|\g_j|\le \sum_{m=t-2j}^{t-k}{t-k\choose m}{3k\choose \big\lceil\frac 32 (t-m)\big\rceil}k^{j-\lceil (t-m)/2\rceil},$$
where 
the inequality is due to Lemma~\ref{lemkeyred}~(iii) that states that a subfamily of $\T_k(X)$ cannot be  $>k$-spread for a set $X$ of size at most $t+j-1$ (applied for $\g_j(X)$), and Observation~\ref{obs33}. Since $k\le t^{0.01}$, it is easy to see\footnote{Doing the same analysis as after \eqref{eqfj}} that the maximum in the sum above is attained when $m = t-2j$, and, moreover, the following holds
$$|\g_j|\le (1+o(1)){t-k\choose t-2j}{3k\choose 3j}k^{j-\lceil (t-(t-2j))/2\rceil}= (1+o(1)){t-k\choose t-2j}{3k\choose 3j}.$$

It is a straightforward calculation to see that for any $0\le j<k$
$${t-k\choose t-2j}{3k\choose 3j}\le \frac{k^4}t {t\choose t-j}=o\Big( t^{-0.9}{t\choose t-j}\Big).$$
From here, using Lemma~\ref{lemsize}, we conclude that
for each $j=0,\ldots, k-1$ we have
$$|\Sigma_n[\g_j]|\le (1+o(1)){t-k\choose t-2j}{3k\choose 3j}(n-t-j)!\le t^{-0.9}|\aaa_j|,$$
and thus
\begin{equation}\label{eqres3}
  |\Sigma_n[\T_{k}\setminus \W_k]|\le \sum_{j = 0}^{k-1}|\Sigma_n[\g_j]|\le t^{-0.9} \sum_{j = 0}^{t^{0.01}}|\aaa_j|\le t^{-0.8}\max_j |\aaa_j|.
\end{equation}

From inequalities \eqref{eqres1},\eqref{eqres2} and \eqref{eqres3}, together with the bound on $|\ff'|$ that is given by the application of Theorem~\ref{thmstage2}, we conclude that, as long as $|\ff| \ge \frac 12 \max_x |\aaa_x|$, all but an $o(1)$-fraction of all sets in $\ff$ are contained in $\Sigma_n(\W_k)$. Using the last conclusion of Observation~\ref{obs7}, all but a $o(1)$-fraction of the sets in $\W_k$,  are contained in the union $A\cup B$, where $A,B\in \W_k$ and $|A\cap B| = t$. (cf. the paragraph after \eqref{eqres2}) That is, they belong to ${A\cup B\choose t+k}$, where $|A\cup B| =t+2k$. In other words, up to isomorphism, all but a $o(1)$-fraction of the sets in $\ff$ belong to $\aaa_k$.


The last step is to show that $\ff$ must be contained in $\aaa_k$. 
Assume that there is a permutation $\sigma\in\ff$ such that $|\sigma\cap (A\cup B)|\le t+k-1$. Consider the family of partial permutations $\g:=\{F\in {A\cup B\choose t+k}: |F\cap \sigma|\le t-1\}$ and $\h:={A\cup B\choose t+k}\setminus \g$.

First, we note that $|\g| \ge {t+k-1\choose k}= (1+O(t^{-0.9})){t+2k\choose k},$ since $k\le t^{0.01}$. We of course have $|\h|= O(t^{-0.9}){t+2k\choose k}$. From here, it should be clear that $|\ff|=(1+o(1)) |\ff[\g]|\le (1+o(1))\sum_{X\in \g}|\ff(X)|$. Let us bound $|\ff(X)|$ for each specific $X$. 
Apply a suitable isomorphism that maps $X$ to a partial permutation $\pi$: $\pi(i)=i$ for each $i= n-t-k+1,\ldots, n$. Then $\sigma$ induces a certain partial permutation $\sigma'$ on $[n-t-k]$, and any $\sigma'$-derangement (that is, any permutation that does not intersect $\sigma'$) on $[n-t-k]$ cannot 
appear in $\ff(X)$. The number of partial derangements is at least the number of derangements, which is at least $\frac 1e(n-t-k)!$. Therefore, $|\ff(X)|\le \big(1-\frac 1e\big)|\Sigma_n(X)|$, and thus
$$|\ff|\sim |\ff[\g]|\le \sum_{X\in \g}|\ff(X)|\le \sum_{X\in \g}\big(1-\frac 1e\big)|\Sigma_n(X)|\sim \big(1-\frac 1e\big)|\aaa_k|,$$
where the last asymptotic equality is due to the asymptotic of $|\g|$ and Lemma~\ref{lemsize}. We conclude that $|\ff|\le \big(1-\frac 1e+o(1)\big)|\aaa_k|$, unless $\ff$ is contained in $\aaa_k$. This completes the proof of the theorem.

\end{document}